\documentclass{amsart}
\usepackage{amsmath}
\usepackage{amsthm}

\usepackage[pdftex,bookmarks=false,colorlinks,
linkcolor=black,citecolor=black]{hyperref} 

\newtheorem{theorem}{Theorem}[section]
\newtheorem*{plaintheorem}{Theorem}
\newtheorem{lemma}[theorem]{Lemma}
\newtheorem{proposition}[theorem]{Proposition}

\newtheorem*{problem}{Problem}
\theoremstyle{definition}
\newtheorem{definition}[theorem]{Definition}
\theoremstyle{remark} 
\newtheorem{remark}[theorem]{Remark}
\newtheorem{example}[theorem]{Example}

\newcommand{\bbr}{\mathbb{R}} 
\newcommand{\bbz}{\mathbb{Z}}
\newcommand{\bbq}{\mathbb{Q}}
\newcommand{\bbc}{\mathbb{C}}
\newcommand{\pp}{\mathcal{P}}

\newcommand{\ugh}{R_{\bar{r}}}
\newcommand{\ughs}{R_{\bar{s}}}

\title{Diophantine Correct Open Induction}
\author{Sidney Raffer}
\address{ Phnom Penh, Cambodia} 
\email {sraffer@gmail.com}
\thanks{ I am  grateful to my advisor Attila M\'at\'e, and to
Roman Kossak for his kind assistance over many years. The material
here is based in part on my dissertation \textit{ Diophantine Properties
of Ordered Polynomial Rings, } submitted to the City University of New
York, June 2000.}
\begin{document}
\begin{abstract}We give an induction-free 
axiom system for diophantine correct open induction. We reduce the
problem of whether a finitely generated ring of Puiseux polynomials is
diophantine correct to a problem about the value-distribution of a
tuple of semialgebraic functions with integer arguments. We use this result, and a
theorem of Bergelson and Leibman on generalized polynomials, to
identify a class of diophantine correct subrings of the field of
descending Puiseux series with real coefficients.
\end{abstract}

\maketitle

\section*{Introduction}

\subsection*{Background}
 A model of open induction is a discretely ordered ring whose semi\-ring
 of non-negative elements satisfies the induction axioms for
 open\footnote{A formula is ``open''  if it is quantifier-free.}  formulas.

 Equivalently, a model of open induction is a discretely
 ordered ring $R$, with real closure $F$, such that every element of
 $F$ lies at a finite distance from some element of $R$.\footnote{Consequently,
 the inequality $r\le x<r+1$ defines a
 function $r=\lfloor x \rfloor$ from $F$ onto $R$. This function is the natural counterpart of the usual integer part operator from $\mathbb{R}$ onto $\mathbb{Z}$.}

The surprising equivalence between these two notions was discovered by
Shepherdson \cite{Sh}. This equivalence enabled him to identify
naturally occurring models of open induction made from Puiseux
polynomials. Let $F$ be the field of descending\footnote{A descending Puiseux series with real
coefficients has the form $\sum_{i<M}a_it^{i/D}$, where $M$ is an
integer, $D$ is a positive integer, and the $a_i$ are real. } Puiseux series with
coefficients in some fixed real closed subfield of $\bbr$.  Puiseux's
theorem implies that $F$ is real closed. There is a unique ordering on
$F$, in which the positive elements are the series with positive leading
coefficients. Define an ``integer part'' function on $F$
as follows:
$$\left\lfloor \sum_{i<M} a_it^{i/D}  \right\rfloor= \lfloor
 a_0\rfloor +  \sum_{i>0} a_it^{i/D}$$
where $\lfloor a_0\rfloor$ is the usual integer part of the
 real number $a_0$.

The image of $\lfloor\,\cdot\,\rfloor$ is the subring $R$ of $F$
 consisting of all Puiseux polynomials with constant terms in
 $\bbz$. Since every Puiseux series is a finite distance from its
 leading Puiseux polynomial, it is immediate that every element of $F$
 is a finite distance from some element of $R$. The discreteness of
 the ordering on $R$ is a consequence of the 
 polynomials in $R$ having integer constant terms. By Shepherdson's
 equivalence, $R$ is a model of open induction.

There has been some effort to find other models of open induction in
the field of real Puiseux series $F$, satisfying additional properties of
the ordered ring of integers.  Perhaps the most extreme possibility in
this regard is that $F$ contains a model of open induction that is
diophantine correct. We shall say that an ordered ring is {\it
diophantine correct} if it satisfies every universal sentence true in
the ordered ring of integers.  We refer to the theory of diophantine
correct models of open induction as $DOI$. To make this notion
precise, we shall assume that ordered rings have signature
$(\,+\,-\,\cdot\,\le\,0\,\,1\,)$.  All formulas will assumed to be of
this type.  Diophantine correctness amounts to the requirement that an
ordered ring not satisfy any system of polynomial equations and
inequalities that has no solution in the ring of integers.

Shepherdson's models are not diophantine correct.\footnote{For
example, there are positive solutions of the equation $x^2=2y^2$ {\it via}
the Puiseux polynomials $x=\sqrt{2}t$ and $y=t$.}  However, there are
other models of open induction in the field of real Puiseux series,
notably the rings constructed by Berarducci and Otero \cite{Be}, which
are not obviously not diophantine correct. More generally, it seems to be
unknown whether the field of real Puiseux series has a diophantine
correct integer part.  

\begin{problem} Let $F$ be the field of Puiseux series with
coefficients in a real closed subfield $E$ of $\bbr$ of positive 
transcendence degree over the rationals. Must {\rm(}Can{\rm)} $F$
contain a model of $DOI$ other than $\bbz$?
\end{problem}

We prove in Section 2 that the field $E$ must
have positive transcendence degree, otherwise the only model of $DOI$
contained in $F$ is $\bbz$.
\subsection*{ Wilkie's Theorems and the Models of
Berarducci and Otero}
 Wilkie \cite{Wi} gave  necessary and sufficient
 conditions for an ordinary (unordered) ring $R$ to have an expansion
 to an ordered ring that extends to a model of open induction. These
 conditions are

\begin{enumerate}
\item For each prime $p$, there must
be a homomorphism $h_p:R\to \bbz_p$, where $\bbz_p$ is the ring of
$p$-adic integers.\footnote{This is equivalent to the condition that
for every positive integer $n$ and every prime $p$ there is a
homomorphism from $R$ onto the ring $\bbz/p^n\bbz$.}
\item It must be possible to discretely order the ring $R$.
\end{enumerate}
These conditions are independent. For example, the ring
$R=Z[t,(t^2+1)/3]$ is discretely ordered by making $t$ infinite.\footnote{To prove discreteness, first show that $R/3R$ is a
nine-element field. If $H$ is a polynomial with integer coefficients
and if $r=H(t,(t^2+1)/3)$ is finite but not an integer, then $r$ has the
form $a/3^n$, where $3\not|\,a$ and $n>0$. Map the equation
$3^nH(t,(t^2+1)/3)=a$ to $R/3R$ to get a contradiction.} But the equation
$1+x^2=3y$ is solvable in $R$ but not in $\bbz_3$, so there is no
homomorphism from $R$ into $\bbz_3$.

 Conversely,  let $g(t)$ be the polynomial
$(t^2-13)(t^2-17)(t^2-221)$. The ring $R=\bbz[t,t+1/(1+g(t)^2)]$
can be mapped homomorphically to $\bbz_p$ for every $p$,\footnote{To find a homomorphism $h$ from $R$ into
$\bbz_p$, use the fact that the polynomial $g(x)$ has $p$-adic zeros
for all $p$. See \cite{Bo}. Set $h_p(x)=r$, where $r$ is a $p$-adic zero of
$g$, and set $h_p(x+1/(1+g^2))=r+1$, and show that $h_p$ extends to a homomorphism from $R$
into $\bbz_p$.  } but cannot be
discretely ordered: The second generator minus the first is between
two integers if $t$ is not.

Wilkie \cite{Wi} gave conditions under which an ordered ring can be extended so as to
preserve these two conditions  (using a single ordering.) We
paraphrase his results.
\begin{plaintheorem}[\textrm {Wilkie's Extension Theorem}]
Let $R$ be discretely ordered ring. Suppose that for every prime $p$
there is a homomorphism $h_p:R\rightarrow \bbz_p$. Let $F$ be a
real closed field containing $R$ and let $s\in F$. Then
\begin{enumerate}
\item  If $s$ is not a finite distance from any element of $R[\bbq]$, and $s$ is not
infinitely close to any element of the real closure of $R$ in $F$, then
$R[s]$ is discretely ordered as a subring of $F$, and the homomorphisms $h_p$ can be
extended to $R[s]$ by assigning $p$-adic values to $s$ arbitrarily.
\item If $s\in R[\bbq]$, then
choose $n\in \bbz$ so that $ns\in R$. Choose $m\in \bbz$ so that $n$ divides $h_p(ns)-m$ in
$\bbz_p$, for every prime $p$. Put $r=(ns-m)/n$. Then $R[r]$ is
discretely ordered, and the homomorphisms $h_p$ extend to $R[r]$ {\it via} $h_p(r)=(h_p(ns)-m)/n$.
\end{enumerate}
\end{plaintheorem}

The choice of $m$ in Case (2) is always possible because $n$ will be a
unit in $\bbz_p$ for all $p$ prime to $n$. Suppose $n$ has prime
decomposition 
$\prod p_i^{e_i}$. For each of the
 prime divisors $p_i$ of $n$,  choose $m_i\in \bbz$ so close\footnote{In the
 sense of the $p$-adic metric.} to $h_{p_i}(ns)$
 that $m_i\equiv h_{p_i}(ns)\mod
 p_i^{e_i}$. Then use the Chinese remainder theorem to get
 $m\equiv m_i\mod p^{e_i} $.

The point is that starting with an ordered ring $R$ and homomorphisms
$h_p$ as above, one can extend $R$ to a model of open induction by
repeatedly adjoining missing integer parts of elements of a real
closure of $R$. We give an example of how this is done. Let
$R=\bbz[t]$, and let $F$ be the field of real Puiseux series. Let
$h_p:R\to\bbz_p$ be the homomorphism given by the rule\footnote{$h_p$
is the unique homomorphism from $R$ into $\bbz_p$ taking $t$ to
$1/(1-p)=1+p+p^2+\ldots$.}
$$h_p(f(t))=f\left(1+p+p^2+\ldots\right).$$ Think of the polynomial $s=t^2/36$
as an element of some fixed real closure of $R$. Then $s$ has no integer
part in $R$.  We shall adjoin an integer part {\it via} Case
(2). Since $36s\in R$, we must find $m\in\bbz$ so close to
$h_p(36s)=1+2p+3p^2+\ldots$ that 36 will divide $h_p(36s)-m$. This is
only an issue for $p=2,3$, since otherwise 36 is a unit. It is enough
to solve the congruences
\begin{align*} m&\equiv 1+2\cdot2^1\mod 2^2\\m&\equiv 1+2\cdot 3^1\mod 3^2.\end{align*}
Here $m=25$ does the job. Thus we adjoin $(36s-25)/36=(t^2-25)/36$ to $R$.

To continue, the element $\sqrt{2}t$ is not within a finite distance
of any element of the ring $R_1=\bbz[t,(t^2-25)/36]$. We can fix that
{\it via} Case (1) by adjoining $\sqrt{2}t+r$, where $r$ is any
transcendental real number.  The fact that $r$ is transcendental
insures that $\sqrt{2}t+r$ is not infinitely close to any element of
the real closure of $R_1$. We can extend the maps $h_p$ to $R_1$ by
assigning  $p$-adic values to $\sqrt{2}t+r$ arbitrarily.

The models of open induction in \cite{Be} are constructed, with some careful
bookkeeping, by iterating the procedure just described.  Up to
isomorphism, the result is a polynomial ring $R$ over $\bbz$ in infinitely many
variables that becomes a model of open induction by  adjoining elements
$r/n$ $(r\in R, n\in\bbz)$ in accordance with Case (2) of Wilkie's
extension theorem.  We  suspect that all of these rings are
diophantine correct. As we shall see, the question turns on how subtle
are the polynomial identities  that can hold on the integer points of a
certain class of semialgebraic sets.

The plan of the paper is as follows. In Section 1 we give a simplified
axiom system for $DOI$. In Section 2 we give number-theoretic
conditions for a finitely generated ring of Puiseux polynomials to be
diophantine correct: We show how the diophantine correctness of 
such a ring is a problem about the distributions of the values at
integer points of certain tuples of generalized polynomials.\footnote{A generalized polynomial is an expression made from
arbitrary compositions of real polynomials with the integer part
operator. See \cite{Bl}.} In Section 3 we give some recent results on
generalized polynomials, and in Section 4 we use these results to give a class of
ordered rings of Puiseux polynomials for which consistency with the
axioms of open induction and diophantine correctness are equivalent.

\section{ Axioms for $DOI$}

In this section we prove that $DOI$ is
equivalent to all true (in $\bbz$) sentences $\forall \bar{x}\exists
y\phi$, with $\phi$ an open formula. The underlying reason for this
fact is that compositions of the integer part operator with
semialgebraic functions suffice to witness the existential quantifier
in every true $\forall \bar{x}\exists y$ sentence.

\begin{theorem}  $DOI$ is axiomatized by the set of all
sentences true in the ordered ring of integers of the form $\forall x_1\forall
x_2\ldots \forall x_n\exists y\phi$, with $\phi $ an open formula.
\end{theorem}

The proof requires two lemmas. The first is a parametric
version of the fact that definable subsets in real closed fields are
finite unions of intervals. 

Let $F$ be a real closed field and $\phi (x,\bar{y})$ a formula. For
each $\bar{r}\in F,$ the subset of $F$ defined by $\phi (x,\bar{r})$
can be expressed as a finite union $I_{1,\bar{r}}\cup \ldots \cup
I_{n,\bar{r}},$ where the $I_{i,\bar{r}}$ are either singletons or
open intervals with endpoints in $ F\cup \{\pm \infty \}$.  We shall
require the fact that for each $\phi$ there are formulas
$\gamma_i(x,\bar{y})$ such that for every $\bar{r}$, the $\gamma_i(x,\bar{r})$ define such
intervals $I_{i,\bar{r}}$.

\begin{lemma} 
Let $\phi (x,\bar{y})$ be a formula in the language of ordered rings. Then
there is a finite list of open formulas $\gamma _i(x,\bar{y})$ such
that the theory of real closed fields proves the following sentences: 

\begin{enumerate}
\item 
 $ \forall x,\bar{y}\,(\phi (x,\bar{y})\leftrightarrow%
\bigvee\nolimits_i\gamma _i(x,\bar{y}))$\smallskip
\item 
$\bigwedge\nolimits_i\forall \bar{y}\,\,\,(
(\neg \exists x\,\gamma _i(x,\bar{y}))\,\,\,\vee $

\qquad $(\exists !x\,\gamma _i(x,\bar{y} ))\,\,\,\,\vee$

\qquad  $(\exists z\forall x(\gamma _i(x,\bar{y})\leftrightarrow
x<z))\,\,\,\vee $

\qquad$(\exists z\forall x(\gamma _i(x,\bar{y})\leftrightarrow
x>z))\,\,\,\vee$

\qquad$(\exists z,w\,\forall x\,(\gamma _i(x,\bar{y} 
)\leftrightarrow z<x<w)))$
\end{enumerate}
\end{lemma}

Formula (1) asserts that for any  tuple
$\bar{r}$ in a real closed field, the set defined by $\phi(x,\bar{r})$ is the union of the
sets defined by the $\gamma_i(x,\bar{r})$. Formula (2) asserts that
each set defined by $\gamma_i(x,\bar{r})$ is either empty, or a
singleton, or an open interval. 
\begin{proof} This is a well-known consequence of Thom's Lemma. See \cite{Vd}. 
\end{proof}

The next Lemma  shows that in models of
$OI$, a one-quantifier universal formula is equivalent to an
existential formula.

\begin{lemma}
For every formula $\forall x\,\phi (x,\overline{y})$ with $\phi $
open, there are open formulas $\psi _i(x_i,\overline{y})$ such that 
$$
OI\vdash \forall \overline{y}\,((\forall x\phi (x,\overline{y}%
))\leftrightarrow \bigwedge\nolimits_i\exists x_i\psi _i(x_i,\overline{y})).
$$
\end{lemma}

The idea of the proof is as follows: If the formula $\forall x\phi
(x,\overline{r})$ holds in some model $R$ of open induction, with
$\bar{r}\in R$, then the formula $\phi (x,\overline{r})$ must hold for
all elements $x$ of the real closure of $R$, except for finitely many
intervals $U_i$ of length at most 1. The existential formula $\exists
x_i\psi _i(x_i,\overline{y})$ says that for some $e_i\in R$, the set
$U_i$ is included in the open interval $(e_i,e_i+1)$.

\begin{proof}[Proof of Lemma 1.3]
Let $\gamma_i$ be the formulas given by the statement of
Lemma 1.2, using $\neg\phi$ in place of $\phi$. Thus Formula (1) of Lemma
1.2 now reads
\begin{equation*} \forall x,\bar{y}\,(\neg\phi (x,\bar{y})\leftrightarrow%
\bigvee\nolimits_i\gamma _i(x,\bar{y})).\tag{$*$}
\end{equation*}

By Tarski's Theorem, choose quantifier free formulas $\alpha _i(z,\bar{y})$
 and $\beta _i(z,\bar{y})$ such that the theory of real closed fields proves 
$$
 \forall z,\bar{y}\,(\alpha _i(z,\bar{y})\leftrightarrow \forall
 w\,(\gamma _i(w,\bar{y})\rightarrow z<w))  
$$
and 
$$ \forall z,\bar{y}\,(\beta _i(z,\bar{y})\leftrightarrow \forall
 w\,(\gamma _i(w,\bar{y})\rightarrow w<z)). 
$$

If $F$ is a real closed
 field, and if $\bar{r}\in F,$ then $\alpha
 _i(x_i,\bar{r}) $ defines all elements $x_i$ of $F$ such that $x_i$
 is less than any
 element of the set defined by $\gamma _i(x,\bar{r})$. Similarly, $\beta
 _i(x_i,\bar{r}) $ defines all elements $x_i$ of $F$ such that $x_i$
 is greater than any
 element of the set defined by $\gamma _i(x,\bar{r})$. 

 Define the formula $\psi _i(x_i,\overline{y})$ required by the conclusion of
 the Lemma to  be 
$$
 \alpha _i(x_i,\bar{y})\wedge \beta _i(x_i+1,\bar{y}). 
$$
We must prove that the equivalence 
$$
\forall \bar{y}\,((\forall x\phi (x,\bar{y}))\leftrightarrow
\bigwedge\nolimits_i\exists x_i\psi _i(x_i,\bar{y})) 
$$
holds in every model of open induction $R$.

For the left-to-right direction, let $\bar{r}$ be a
tuple from $R$, and suppose that $R$ satisfies $\forall
x\,\phi (x,\bar{r})$. For each $i$ we must find  $g$ in $R$
such that 
\begin{equation*}
R\models \alpha _i(g,\bar{r})\wedge \beta _i(g+1,\bar{r}).\tag{$**$}
\end{equation*}

Let $F$ be a real closure of $R$. Let $I_{i,\bar{r}}$ be the open interval of $F$ defined by
the formula $\gamma _i(x,\bar{r}).$ 

The interval $I_{i,\bar{r}}$ cannot be unbounded: It must have both
 endpoints in $F$. Otherwise $I_{i,\bar{r}}$ would meet $R$.\footnote{If
 $R$ is an ordered ring and $F$ is a real closure of $R$, then $R$ is
 cofinal in $F$. \cite{Br}.} If $I_{i,\bar{r}}$
 did meet $R$, then the
 universal sentence ($*$), would give an element $e\in R$ for which
 $\neg\phi(e,\bar{r})$ holds, contrary to our assumption that $R\models \forall x\,\phi
 (x,\bar{r})$.  Therefore $I_{i,\bar{r}}$ is a
bounded open interval.

 If the interval $I_{i,\bar{r}}$ is empty, then  every $g\in R$ will trivially satisfy  condition
$(**)$, and the proof will be complete. Therefore, we can assume that $I_{i,\bar{r}}$ is nonempty. 
Formula ($*$) then implies that the half-open intervals defined by the formulas $\alpha _i(x_i,\bar{r})$
and
$\beta _i(x_i+1,\bar{y})$  will each have  an endpoint in $F$, i.e., they
will not be of the form $(-\infty,\infty)$.

 The least number principle for open induction\footnote{In a model of
open induction $R$, if a non-empty set $S\subseteq R$ is defined,
possibly with parameters, by an open formula, and if $S$ is bounded
below, say by $b$, then $S$ has a least element. Otherwise  if
$s\in S$ then the set of non-negative $x\in R$ such that $x+b\le s$ is
inductive. See [7].} implies that there is a greatest element $g\in R$ such
that $R\models \alpha _i(g,\bar{r})$. The maximality of $g$ implies
that $R\models \neg \alpha _i(g+1,\overline{r})$.  Hence $g+1$ is at
least as large as some element of $I_{i,\bar{r}}$. Since
$I_{i,\bar{r}}$ is disjoint from $R$, it follows that $g+1$ is greater
than every element of $I_{i,\bar{r}}$.  Therefore $\beta
_i(g+1,\bar{r})$ holds in $R$. We have found $g$ satisfying the
required condition ($**$).

For the right-to-left direction of the equivalence, assume that for
every $i$, we have elements $b_i\in R$ such that
$$R\models \alpha _i(b_i,\bar{a})\wedge \beta _i(b_i+1,\bar{a}).$$
This same formula will  hold in $F$, hence for each $i$, 
$$
F\models \forall w\,(\gamma _i(w,\bar{a})\rightarrow b_i<w)\wedge \
\forall w\,(\gamma _i(w,\bar{a})\rightarrow w<b_i+1). 
$$
The last displayed statement asserts that every element $b$ of $F$ satisfying $\gamma _i(b,\bar{a})$ lies
between $b_i$ and $b_i+1$. But no element of $R$ lies between $b_i$
and $b_i+1$. Therefore for every $b\in R$, 
$$R\models \neg \bigvee\limits_i\,\gamma _i\left( b,\bar{a}\right).$$
This assertion, together with ($*$), gives the conclusion $R\models \forall x\,\phi (x,\bar{a})$.
\end{proof}

\begin{proof}[Proof of Theorem 1.1]

Let $T$ be the theory of all sentences true in $\bbz$ of the form $\forall x_1\forall
x_2\ldots \forall x_n\exists y\phi$, with $\phi $ an open formula. We
 prove the equivalence $T\Leftrightarrow DOI$.
 
\bigskip
\noindent $T\Rightarrow DOI:$

It is immediate that $T\Rightarrow $ $DOR+\forall _1(\bbz).$ It remains
to verify that $T$ proves all instances of the induction scheme for open
formulas. For each open formula $\phi ,$ the induction axiom 
$$
\forall \overline{x}\,((\phi (\overline{x},0)\wedge \forall y\ge 0\,(\phi (%
\overline{x},y)\rightarrow \phi (\overline{x},y+1)))\rightarrow \forall z\ge
0\,\phi (\overline{x},z)) 
$$
is logically equivalent to
$$\forall \overline{x}\,\forall z\,\exists y\,(z\ge 0\rightarrow
(y\ge 0\wedge \phi (\overline{x},0)\wedge ((\phi (\overline{x}%
,y)\rightarrow \phi (\overline{x},y+1)\,)\rightarrow \phi (\overline{x},z)))).
$$
The latter belongs to $T.$ \

\bigskip
\noindent $DOI\Rightarrow T:$

 Suppose that $R\models DOI.$
Let $\phi (\overline{x},y)$ be an open formula such that $$\bbz\models
\forall \overline{x}\,\exists y\,\phi (\overline{x},y).$$
We  
prove 
that $R\models \forall \overline{x}\,\exists y\,\phi
(\overline{x},y)$.

By Lemma 1.3, there are open formulas $\psi _i$ such that $OI$ proves
the equivalence 
$$
\forall \overline{x}\,((\exists y\,\phi (\overline{x},y))\longleftrightarrow
\bigvee_i\forall z_i\psi (\overline{x},z_i)). 
$$

The last two displayed assertions imply that 
$\bbz\models
\forall \bar{x}\bigvee_i\forall z_i\psi (\overline{x},z_i))$.
But $R$ is diophantine correct, therefore $R\models
\forall \bar{x}\bigvee_i\forall z_i\psi (\overline{x},z_i))$. 
Since $R$ is a model of $OI$, the above equivalence  holds in
$R$. Therefore $R\models \forall 
\bar{x}\,\exists y\,\phi (\bar{x},y)$.  
\end{proof}

\section{ Diophantine Correct Rings of Puiseux Polynomials}

Let $\pp$ denote the ring of Puiseux polynomials with real 
coefficients.  We will think of Puiseux polynomials interchangeably as
formal objects and as functions from the positive reals to the
reals. The following theorem describes the conditions for a
finitely generated subring of $\pp$ to be diophantine correct, in
terms of the coefficients of a list of generating polynomials. We
shall use this theorem to investigate the diophantine correct subrings
of $\pp$. To simplify notation  we temporarily assume  that not all the
 coefficients of the generating polynomials are algebraic
 numbers.

\begin{theorem}
Let $f_1\ldots f_n$ $\in \pp $. Assume that the $f_i$ are 
non-constant, and that the field $F$ generated by the coefficients of the
$f_i$ has transcendence degree at least 1 over $\bbq$.  Let $\bar{r}=r_1\ldots r_l$
be a transcendence basis for $F$ over $\bbq$. Then
\begin{enumerate}
\item  There is an open formula $\theta (x_1\ldots x_l,y_1\ldots y_n)$
such that $\theta (\bar{r},\bar{y})$ holds in $\bbr$ at $\bar{y}$ if and
only if for some real $t\ge 1$, $\bigwedge_iy_i=f_i(t).$
\item Choose $\theta$ as in {\rm (1)}. The ring $\bbz[\bar{f}]$ is
diophantine correct if and only if for every open neighborhood
$U\subseteq \bbr^l$ of $\bar{r}$ and for every positive integer $M$,
there are points $\bar{s}\,\in U$ and integers $\bar{m}$ such that
$\min_i|m_i|>M$ and $\bbr\models\theta (\bar{s},\bar{m})$.
\end{enumerate}
\end{theorem}

We give two examples to show how  Theorem 2.1 can be used to determine whether a given subring of $\pp$ is diophantine correct.

\begin{example} Let $R=\bbz[t,f(t)-r_1]$, where $r_1$
is a real transcendental and $f$ is a polynomial with algebraic
coefficients. The formula $\theta(r_1,y_1,y_2)$ expresses the
condition
$$\exists t\ge 1\,(y_1=t\wedge y_2=f(y_1)-r_1).$$
Eliminating the quantifier we obtain\footnote{For the sake of clarity
we  neglect the translation into the language of ordered rings.}
$$\theta(r_1,y_1,y_2): y_1\ge 1\wedge f(y_1)-y_2=r_1.$$

It follows that  the ring $R$ is diophantine correct if and only if
there are positive integers $\bar{y}$  making
$f(y_1)-y_2$ arbitrarily close to $r_1$.

It is known\footnote{This is a consequence of Weyl's Theorem on uniform distribution. See \cite{Ca}, p71.}  that the values of $f(y_1)-y_2$ are either dense in
the real line, if $f$ has an irrational coefficient other than its
constant term, or otherwise discrete.  In the former case $R$ is
diophantine correct. In the latter case $f(y_1)-y_2$ could only
approach $r_1$ by being equal to $r_1$, which is impossible since
$r_1$ is transcendental.

\end{example}             
\begin{example} Let $R=\bbz[t,\sqrt{2}t-r, 2\sqrt{2}rt-s]$, with
$r$ and $s$ algebraically independent. Then
$R$ is diophantine correct if and only if the point
\begin{equation*}
(\sqrt{2}y_1-y_2, 2\sqrt{2}y_1(\sqrt{2}-y_2)y_1-y_3)\tag{$*$}
\end{equation*} 
can be made arbitrarily close to $(r,s)$. This is a non-linear
approximation problem, and there is no well-developed theory of such
problems. In this case the identity
$$(\sqrt{2}y_1-y_2)^2=(2\sqrt{2}y_1(\sqrt{2}-y_2)y_1-y_3)-(2x^2-y^2-y_3)$$
implies that the point ($*$) cannot tend to the pair $(r,s)$ unless
$r^2-s$ is an integer. Hence the requirement that $r$ and $s$ be algebraically
independent cannot be met.
\end{example}

The most general case of Theorem 2.1 cannot be written down 
explicitly, because the algebraic relations between coefficients can
be arbitrarily complex. But, following the notation of Theorem 2.1,  the fact that the $r_i$ are
algebraically independent implies that in the relation
$\theta(\bar{x},\bar{y})$, if $\bar{x}$ is restricted to a small
enough neighborhood of $\bar{r}$ then each $x_i$ is a
 semialgebraic function of
 $\bar{y}$.\footnote{See \cite{Vd}, p32, Lemma 1.3.} Therefore the  problem of whether a
finitely generated ring of Puiseux polynomials is diophantine correct
always has the form: ``Are there tuples of integers $\bar{y}$ such that
the points $(\sigma_1(\bar{y}),\ldots,\sigma_n(\bar{y}))$ tend to the
point $\bar{r}$?'' where the $\sigma_i$ are semialgebraic functions.

This general type of problem is undecidable, since it contains Hilbert's tenth
problem.\footnote{For example,
$f(y_1,\ldots,y_{n-2})^2+(\sqrt{2}y_{n-1}-y_n)^2$ can be made
arbitrarily close to a given number $r$  between 0 and 1 if and only if $f$ has an
integer zero.} But the rings that we actually want to use to construct
models of open induction have a  special form, which leads to a
restricted class of problems that may well be decidable. (See Section 3.)

We return to Theorem 2.1, and the conditions for $\bbz[\bar{f}]$ to be
diophantine correct. The idea of the proof is to think of the
polynomials $f_i(t)$ as functions of both $t$ and $\bar{r}$. If $\phi$ is
an open formula, then the statement that $\phi(\bar{f})$ holds in
$\bbz[\bar{f}]$ can be expressed as another open formula
$\psi(\bar{r})$. The latter must hold on an entire neighborhood of
$\bar{r}$, since the $r_i$ are algebraically independent. If
$\bbz[\bar{f}]\models \phi(\bar{f})$ then we can try to perturb the $r_i$ a tiny
bit for very large $t$ so as to make the values $f_i(\bar{r},t)$ into
integers. The formula $\theta$ expresses the relation between the
perturbed values of $\bar{r}$ and the resulting integer values of
$\bar{f}$.

  We hope that this explanation motivates the use of following three Lemmas. We
 omit the straightforward proofs.

\begin{lemma} 
Let $f_1,f_2\ldots f_n\in \pp$. Let $\phi (\bar{x})$ be an open
formula.  Then $\bbz[\bar{f}]\models \phi (\bar{f})$ if and
only if for
all sufficiently positive $t\in \bbr$, the
formula $\phi (\bar{x})$ holds in $\bbr$ at the tuple of real numbers $\bar{f}(t)$. \qed
\end{lemma}

\begin{lemma}
Let $\bar{f}=f_1(t)\ldots f_n(t)\in \pp$. The
ordered ring $\bbz[\bar{f}]$ is diophantine correct if and only if for every open
 formula $\phi (\bar{x})$ such that $\bbz[\bar{f}]\models \phi
 (\bar{f}),$ there exists $\bar{m}\in \bbz$ such that
 $\bbz\models \phi(\bar{m})$. \qed
\end{lemma}

\begin{lemma}
Suppose that $\phi (\bar{x})$ is a formula and $\bar{r}\in \bbr^n$
is a tuple of algebraically
independent real numbers.\footnote{Algebraically independent over $\bbq$.} If $ \bbr\models \phi (\bar{r}),$ then there
is a neighborhood $U$ of $\bar{r}$  such that  for every $\bar{u}\in U$,
$\bbr\models\phi(\bar{u})$. \qed
\end{lemma}

\begin{proof} [Proof of Theorem 2.1]

To prove Part (1), let $f_i(t)=g_i(\bar{c},t)$, where $g_i$ is a
polynomial with integer coefficients, and the $c_i$ are algebraic over
the field $ \bbq(\bar{r})$. Then  $c_i$ can be defined from the
$r_i$, say by a  formula $\gamma_i(\bar{r},\bar{c})$.   Eliminate 
quantifiers from the formula
$$\exists t\ge 1\,\exists \bar{w}\,\left(
y_i=g_i(\bar{w},t)\wedge \gamma_i(\bar{w},\bar{x})\right)$$ to obtain
an open formula $\theta_i(\bar{x},y_i)$, and let $\theta$ be the
conjunction of the $\theta_i$.

To prove the left-to-right direction of Part (2), assume that $\bbz[\bar{f}]$ is diophantine
 correct, and let $\bar{r}$ and $\theta $ be as in Part (1). Let $U\subseteq \bbr^l$
 be an open set containing $\bar{r}$, and fix a positive integer $M$.  We must find $\bar{s}%
\in U$ and $\bar{m}\in \bbz^n$, with  $|m_i|>M$, such that
$\theta (\bar{s},\bar{m})$ holds in $\bbr$.

Since $U$ is open, there is a formula $\gamma (\bar{x})$ which
holds at $\bar{r}$, and which defines an open set  included in $U$. By Tarski's theorem, there
is an open formula $\theta _1(\bar{y})$ such that 
$$
RCF\vdash \,\,\theta _1(\bar{y})\leftrightarrow \exists \bar{x}\,((\bigwedge_i\left|
y_i\right| >M)\wedge \gamma (\bar{x})\wedge \theta (\bar{x},\bar{y})).
$$

The formula $\theta _1(f_1(t),\ldots ,f_n(t))$ must hold in $\bbr$
for all sufficiently large $t$, since  the functions $\left| f_i(t)\right| $ tend to infinity with
$t$, and since moreover we can witness the above existential quantifier with
$\bar{r}$. Therefore, by Lemma 2.4, $ \bbz[\bar{f}]\models \theta _1(\bar{f})$.

Since $\bbz[\bar{f}]$ is diophantine correct, it follows that there
are integers $\bar{m}\in \bbz^n$ satisfying  
$\theta _1(\bar{y})$. Substituting $\bar{m}$ for $\bar{y}$ in the
above equivalence, the right hand side gives a tuple
$\bar{s}\in \bbr$ such that
$$\bbr\models ((\bigwedge_i\left|
m_i\right| >M)\wedge \gamma (\bar{s})\wedge \theta
(\bar{s},\bar{m})). $$

Since $\gamma(\bar{x})$ defines a subset of $U$, the displayed
statement confirms that $\bar{s}$ and $\bar{m}$ are the tuples required.

To prove the right-to-left direction of Part (2), let $\phi$ be an open
formula such that $\bbr[\bar{f}]\models \phi(\bar{f})$. We  prove
that there are  integers $\bar{m}$ such that $\phi(\bar{m})$ holds in
$\bbz$. It will follow immediately from Lemma 2.5 that
$\bbz[\bar{f}]$ is diophantine correct.

Since $\phi $ is open and since $\phi (\bar{f})$ holds in
 $\bbz[\bar{f}]$, it follows from  
 Lemma 2.4  that $\phi (f_1(t),\ldots ,f_n(t))$ holds in $\bbr$
 for all sufficiently positive $t$. Choose $k>1$ such that $\phi
 (f_1(t),\ldots ,f_n(t))$ holds in $\bbr$ for $t> k$.

The set of points $(f_1(t),\ldots ,f_n(t))$ with $1\le t\le k$ is
bounded. Therefore we can choose $M\in\bbz$ so large that if $t\ge 1$
and if $\min_i \,\left| f_i(t)\right| >M$, then $t>k$. For this choice
of $M$, the formula 
$\psi (\bar{x})$ will hold in $\bbr$ at $\bar{r}$, where $\psi (\bar{x})$ is the formula 
$$
\forall \bar{y}\,\,((\theta (\bar{x},\bar{y})\wedge (\bigwedge_i\left|
y_i\right| >M))\rightarrow \phi (\bar{y})). 
$$
By Lemma 2.6, the subset of $\bbr^l$ defined by $\psi(\bar{x})$ must
include a neighborhood $U$ of $\bar{r}$. By  hypothesis, we can choose 
 $\bar{s}\in U$ and $\bar{m}\in \bbz^n$ so that 
 $$\bbr\models\theta (\bar{s},\bar{m})\wedge \bigwedge_i\left| m_i\right| >M.$$
Instantiating the universal quantifier in $\psi(\bar{s})$ with
$\bar{m}$,  we conclude  that
$\phi (\bar{m})$ holds in $\bbz$. 
\end{proof}

\begin{remark} If the $f_i$ have algebraic coefficients, then
$\bbr[\bar{f}]$ is diophantine correct if and only if  there is a
sequence of real numbers $u_i$ tending to infinity such that
$\bar{f}(u_i)\in \bbz^n$. To prove this, one takes the transcendence basis $\bar{r}$ to empty in the
proof of Theorem 2.1 and one follows the proof, making all the necessary
changes.  This case is not important for our purposes because of the
following fact.
\end{remark}
\begin{proposition}
There are no  models of $DOI$ of transcendence degree one.
\end{proposition}
\begin{proof} Suppose by way of contradiction that $R$ is a model of
$DOI$ of transcendence degree 1. 
 Let $a$ be a non-standard element of $R$. Let $b\in R$ be an integer part
of $\root{3} \of {2}a$.  Then there is a nonzero polynomial $p$ with integer
coefficients such that $p(a,b)=0$. We can assume that $p$ is
irreducible over the rationals. Since $R$ is diophantine correct, the
equation $p(x,y)=0$ must have infinitely many standard solutions. We
 shall prove that this is impossible.

Write $p=p_0+\ldots+p_n$, where $p_i$ is homogeneous of degree $i$,
and $p_n\ne0$. Then $p(a,b)$ has the form $\sum_{i=0}^np_i(1,b/a)a^i.$

Observe that $b/a$ is finite, in fact infinitely close to $\root 3\of
2$, hence all the values $p_i(1,b/a)$ are finite.  It follows that for
$p(a,b)$ to be zero, ${p_n(1,b/a)}$ must be infinitesimal; otherwise
$p_n(1,b/a)a^n$ would dominate all the other terms $p_i(1,b/a)a^i$,
and then $p(a,b)$ could not even be finite.

Since  $b/a$ is infinitely close to
$\root 3\of 2$, it follows that $p_n(1,\root 3\of 2)=0$.
 Since $p_n$ has
integer coefficients, the polynomial $y^3-2$ must divide
$p_n(1,y)$. It follows that $y^3-2x^3$ divides $p_n$.

But if $f(x,y)$ is any polynomial
with integer coefficients irreducible over the rationals, and if
$f$ has infinitely many integer zeros, then the leading
homogeneous part of $f$ must be a constant multiple of a power of a linear or quadratic
form.\footnote{See \cite{Mo}, p266.} This is not the case for $p_n$, thanks to the factor $y^3-2x^3$. Therefore $p$ cannot
have infinitely many integer solutions. This is the required contradiction.
\end{proof}

\section{  Generalized Polynomials}
\subsection*{ Special Sequences of Polynomials} 
We now focus on a restricted class of rings, which arise by adjoining
sequences of integer parts using Wilkie's extension theorem (given in
the Introduction.) A similar but more general type of sequence was defined in \cite{Be} to
construct normal models of open induction.

\begin{definition} A sequence of   polynomials is \textbf{special}
 if it has the form 
 $$f_0(t),f_1(t)-r_1,\ldots,f_n(t)-r_n,$$ where 
\begin{enumerate}
\item $f_0(t)=t$, and the coefficients of $f_1(t)$ are  algebraic.
\item  The $r_i$ are algebraically independent real numbers.
\item For $i>1$, the polynomial $f_i$ has the form $g_i(t,r_1\ldots r_{i-1})$
where $g_i$ is a polynomial with algebraic coefficients.
\end{enumerate}
\end{definition}

Note that  a  ring $\bbz[\bar{f}]$ generated by a special sequence
 contains the polynomial
$t$. As a consequence,  a  polynomial is algebraic over $\bbz[\bar{f}]$ if and only if its
coefficients are algebraic over the field generated by the coefficients
of the $f_i$. 

\begin{example} The sequence of polynomials $t,\sqrt{2}t-r^2,rt-s$, where $r,s$ are
algebraically independent real numbers, is not a special sequence
because $rt$ is not a polynomial in $r^2$ and $t$. The sequence $t, 2t-r,
(r^2+s)t-s$ is  not a special sequence because $(r^2+s)t$ is not
a polynomial in $r$ and $t$.
\end{example}

\smallskip

The conditions for a ring generated by a
special sequence to be diophantine correct can be written out
explicitly.

\begin{proposition}
Suppose that $f_0(t),f_1(t)-r_1,\ldots,f_n(t)-r_n$ is a special
sequence, with $0<r_i<1$. Let $R=\bbz[\bar{f}]$. Choose polynomials
$g_i(t,\bar{r})$  as in Item {\rm(3)} of Definition {\rm3.1}. 

Define the polynomials $\sigma_i$ inductively as
follows. Let $\sigma_1(y_0)=f_1(y_0)$. For $i>1$, let
$$\sigma_i(y_0\ldots y_{i-1})=g_i(y_0,\sigma_1(y_0)-y_1,\ldots,\sigma_{i-1}(y_0\ldots
y_{i-2})-y_{i-1}).$$ Then 

\begin{enumerate}
\item $R$ is diophantine correct if and only if the
system of inequalities
\begin{align*} &\left|\sigma_1(y_0)-y_1-r_1\right|<\epsilon\\
&\left|\sigma_2(y_0,y_1)-y_2-r_2\right|<\epsilon\\
&\quad\quad\ldots\ldots\\ &\left|\sigma_n(y_0,y_1\ldots
y_{n-1})-y_n-r_n\right|<\epsilon\end{align*}
has integer solutions $y_i$ for every positive $\epsilon$. 

\item For all sufficiently small positive $\epsilon$, if $\bar{y}$ is a solution to the inequalities  {\rm (1)} then $y_i=\lfloor \sigma_i(y_0\ldots
y_{i-1})\rfloor$. 
\end{enumerate}
\end{proposition}

\begin{proof} Item (1) simply spells out Theorem 2.1 for  rings generated by 
special sequences. Item (2) follows from the assumption that the $r_i$ are
in the interval $(0,1)$, hence so are the values  $\sigma_i(y_0\ldots
y_{i-1})-y_i$ if $\epsilon$ is small enough. 
\end{proof}

There is another way to think of the inequalities in Proposition 3.3.
Since the equation 
$y_i=\lfloor \sigma_i(y_0\ldots y_{i-1})\rfloor $, holds for all small
enough $\epsilon$, it follows that
$$\sigma_i(y_0\ldots y_{i-1})-y_i=\{\sigma_i(y_0\ldots y_{i-1})\},$$
where $\{\,\cdot\,\}$ is the fractional part operator, defined by
$\{x\}=x-\lfloor x\rfloor$. Replacing $y_1$ with
$\lfloor\sigma_1(y_0)\rfloor$ in the right hand side
of the above equation and continuing in this fashion, we eventually obtain an
expression for $\{\sigma_i(y_0\ldots y_{i-1})\}$ as a function of
$y_0$ alone, where the expression is build from constants and the ring
operations and the fractional and integer part
operators. Following \cite{Bl}, we will call such expressions {\it
bounded generalized polynomials.} The reason for performing this transformation is to relate our questions about diophantine correct rings to a substantial body of results about the distribution of the values of generalized polynomials.

\begin{proposition} Assume $0<r_i<1$. For each system of
polynomial inequalities
\begin{align*} &\left|\sigma_1(y_0)-y_1-r_1\right|<\epsilon\\
&\left|\sigma_2(y_0,y_1)-y_2-r_2\right|<\epsilon\\
&\quad\quad\ldots\ldots\\ &\left|\sigma_n(y_0,y_1\ldots
y_{n-1})-y_n-r_n\right|<\epsilon
\end{align*}

there is an associated system of bounded generalized polynomial inequalities

$$\bigwedge_{i=1}^n|\gamma_i(y_0)-r_i|<\epsilon$$
where the $\gamma_i$ are defined  by
$$\gamma_i(y_0)=\sigma_i(y_0\ldots y_{i-1})-y_i,$$
and the $y_i$ for $i>0$ are defined recursively by
$$y_i=\lfloor \sigma_i(y_0\ldots y_{i-1})\rfloor.$$

Specifically,

\begin{align*} \gamma_1(y_0)&=\{\sigma_1(y_0)\}\\
 \gamma_2(y_0)&=\{\sigma_2(y_0,\lfloor \sigma_1(y_0)\rfloor)\} \\
 \gamma_3(y_0)&=\{\sigma_3(y_0,\lfloor \sigma_1(y_0)\rfloor,\lfloor \sigma_2(y_0,\lfloor \sigma_1(y_0)\rfloor)\rfloor)\} \\
 \ldots\ldots \end{align*}

For all sufficiently small $\epsilon>0$ an integer $y_0$ satisfies the
associated system if and only if there are integers $y_1\ldots y_n$
such that $y_0\ldots y_n$ satisfies the original system.
\end{proposition}

Since open induction is essentially the theory of abstract integer
parts, there is an obvious connection between open induction and
generalized polynomials, yet a systematic study of generalized
polynomials {\it v\'is a v\'is} open induction remains to be done.

There are generalized polynomial identities, that hold for all integers,
such as $$\{\sqrt{2}x\}^2=\{2\sqrt{2}x\{\sqrt{2}x\}\}.$$
Observe that this phenomenon can be explained by the fact that the
ring $$\bbz[t,\sqrt{2}t-r,2\sqrt{2}rt-s],$$ where $r$ and $s$ are
algebraically independent real numbers, does not extend to a model of open
induction. Indeed, we have the identity
$$H(t,\sqrt{2}t-r,2\sqrt{2}rt-s)=s-r^2,$$ where $H(x,y,z)=2x^2-y^2-z$;
so the ring is not discretely ordered.
Substituting $\{\sqrt{2}x\}$ for $r$ and $\{2\sqrt{2}x\{\sqrt{2}x\}\}$
for $s$ one immediately deduces the generalized polynomial identity mentioned above.

Do  all generalized polynomial identities 
arise in this way from ordered rings that violate open induction? 

\subsection*{ Theorems on Generalized Polynomials}
The study of  systems of polynomial inequalities of type
\begin{align*} &\left|\sigma_1(y_0)-y_1\right|<\epsilon\\
&\left|\sigma_2(y_0,y_1)-y_2\right|<\epsilon\tag{$*$}\\
&\quad\quad\ldots\ldots\\ &\left|\sigma_n(y_0,y_1\ldots
y_{n-1})-y_n\right|<\epsilon\end{align*}
goes back at
least to Van der Corput.
He proved

\begin{theorem}[\textrm {Van der Corput \cite{Vc}}]
If  a system of polynomial inequalities of type {\rm(}$*${\rm)} has a solution in integers then
it has infinitely many integer solutions. Moreover, the set $S\subseteq\bbz$ of integers $y_0$ for which there
is a solution $y_0\ldots y_n$ 
is syndetic.\footnote{A subset $S$ of $\bbz$ is syndetic if there are finitely many integers 
$v_i\in \bbz$ such that the union of translates  $\bigcup_iS+v_i$ is
equal to $\bbz$. Equivalently, the gaps between the elements of
$S$ have bounded lengths.}
\end{theorem}

As far we know no one has given an algorithm for the solvability
of arbitrary systems of type ($*$). We believe that  if a system of type ($*$) with real algebraic coefficients has no integer solutions, then this fact is provable
from the axioms of open induction.

By far the most far-reaching results on generalized polynomials are to
be found in Bergelson and Leibman \cite{Bl}.  We paraphrase an important
result from this paper, for use in Section 4.

\begin{theorem} [\textrm {Bergelson and Leibman \cite{Bl}}]
Let $g:\bbz\to\bbr^n$ be a map whose components are bounded
generalized polynomials. Then there is a subset $S$ of $\bbz$ of
density\footnote{Density means here Folner density, defined as
follows. A Folner sequence (in $\bbz$) is a sequence of finite subsets
$s_i$ of $\bbz$ such that for every $n\in\bbz$,
$\lim_{m\to \infty}{|(s_m+n)\Delta s_m|/|s_m|}=0$. Here $\Delta$ means
symmetric difference, and $s_m+n=\{x+n:x\in s_m\}$. A set of integers
$S$ has Folner density zero if $\lim_{n\to\infty}{|S\cap
s_n|/|s_n|}=0$ for every Folner sequence $s_n$.} zero such that the
closure of the set of values of $g$ on the integers not in $S$ is a
semialgebraic set $C$. {\rm(}I.e. $C$ is definable by a formula with real
parameters in the language of ordered rings.{\rm)}  If the coefficients of
$g$ are algebraic then $C$ is definable without parameters.
\end{theorem}

\section{ A Class of Diophantine Correct Ordered
Rings}

The next theorem identifies a class of diophantine
correct ordered rings made from special sequences of polynomials.

\begin{theorem}
 For $i=1\ldots n$ let $g_i(t,x_1,\dots,x_{i-1})$
be  polynomials with algebraic coefficients. For each $n$-tuple of
algebraically independent real numbers $\bar{r}$ such that $0<r_i<1$, 
let $\ugh$ be the ring 
$$\bbz[t,g_1(t)-r_1,g_2(t,r_1)-r_2,\ldots g_n(t,r_1,\ldots,r_{n-1})-r_n]$$
Then
\begin{enumerate}
\item If the ring $\ugh$ extends to a model of open induction for one
algebraically independent $n$-tuple
$\bar{r}$ then it does so for all such  $n$-tuples
$\bar{r}$.
\item If the rings $\ugh$ extend to models of open induction, then
there is an open subset $S$ of the unit box $[0,1]^n$ such that for all
algebraically independent $\bar{r}\in S$, the ring $\ugh$ is diophantine
correct.
\end{enumerate}
\end{theorem}

\begin{proof}[Proof of {\rm(1)}]

Let $\bar{r}$ be an $n$-tuple of real numbers  with algebraically independent
coordinates. Since the ring $\ugh$ is generated by algebraically
independent polynomials, $\ugh$ will extend to a model of open
induction if and only if it is discretely ordered. (See the section on
Wilkie's theorems in the Introduction.) If $\ugh$ is not
discretely ordered, then there is an identity of polynomials in $t$ of the form
$$H(t,g_1(t)-r_1,\ldots,g_n(t,r_1\ldots r_{n-1})-r_n)=K(\bar{r}),$$
where $H$ and $K$ are polynomials and $H$ has integer coefficients. If such an
identity holds for one tuple $\bar{r}$ with algebraically independent
coordinates, then it holds for them all.
\end{proof}

\begin{proof}[Proof of  {\rm(2)}]

The case $n=1$ is done in Example 2.2. We show there that one can take $S$ to be
the interval $(0,1)$. 

Assume $n>1$, and assume that the rings $\ugh$ extend to models of
open induction. The proof will proceed by induction on $n$.

Copying Proposition 3.3, we construct
a sequence of polynomials $\sigma_i$ inductively as follows:  Let $\sigma_1(y_0)=f_1(y_0)$. 
 For $i>1$, let
$$\sigma_i(y_0\ldots y_{i-1})=g_i(y_0,\sigma_1(y_0)-y_1,\ldots,\sigma_{i-1}(y_0\ldots
y_{i-2})-y_{i-1}).$$
Then  the ring $\ugh$ is diophantine correct if and
only if for all positive $\epsilon$ the inequalities 
\begin{align*} &\left|\sigma_1(y_0)-y_1-r_1\right|<\epsilon\\
&\left|\sigma_2(y_0,y_1)-y_2-r_2\right|<\epsilon\\
&\quad\quad\ldots\ldots\tag{$*$}\\ &\left|\sigma_n(y_0,y_1\ldots
y_{n-1})-y_n-r_n\right|<\epsilon\end{align*}
have integer solutions $y_i$.
As in  Proposition 3.4, we define the generalized polynomials
$$\gamma_i(y_0)=\sigma_i(y_0\ldots y_{i-1})-y_i,$$
where $y_i$ is defined inductively by
$$y_i=\lfloor \sigma_i(y_0\ldots y_{i-1}) \rfloor.$$

Then for small enough $\epsilon$ the inequalities
$|\gamma_i(y_0)-r_i|<\epsilon$ hold for $y_0$ if and only if the inequalities
$(*)$ hold for $y_0$ and some choice of integers $y_1\ldots y_n$.

By Theorem 3.5, there is a subset $B$ of $\bbz$ of Folner density 0
 such that the closure of the points $\bar{\gamma}(x)$ for $x\not\in
 B$ is a semialgebraic set $C$ defined over $\bbq$.

If the cell decomposition of $C$ has an $n$-dimensional
cell,\footnote{See \cite{Vd} Chapter 3.} then $C$
contains an open subset of $[0,1]^{n}$ and the theorem is proved.
Otherwise, there is a non-zero polynomial $h$ with integer
coefficients such that 
$$h(\gamma_1(x)\ldots \gamma_{n}(x))=0$$ for all integers $x$ not in
$B$.\footnote{Semialgebraic sets of codimension at least one satisfy
nontrivial polynomial equations. \cite{Vd}}

Our goal is to prove that this is impossible, by showing that if such
an equation held, then $\ugh$ would not be discretely ordered.

By the induction hypothesis there is an open set $S\subseteq[0,1]^{n-1}$ such
that for all points $\bar{s}\in S$ with algebraically independent
coordinates, the rings $\ughs$ are diophantine correct.

 Fix a point $\bar{s}\in S$ with algebraically independent
coordinates. We shall need to know that there are integers $m\not\in B$
for which the point $(\gamma_1(m)\ldots \gamma_{n-1}(m))$ comes
arbitrarily close to $\bar{s}$.

Let $\epsilon>0$. Since $\ughs$ is diophantine correct, Proposition
3.3 Part (1) and Proposition 3.4 imply that there is an integer $m$
such that  
\begin{equation*}|(\gamma_1(m)\ldots \gamma_{n-1}(m))-\bar{s}|<\epsilon.\tag{$**$}
\end{equation*}
By Theorem 3.5 the solutions to ($**$) are syndetic. But no syndetic set has
Folner density zero.\footnote{ Let $s_i$ be the set of integers between $-i$ and $i$. Then  $s_i$ is a
Folner sequence. If $D$ is any syndetic set of integers, then choose $M$ so that $D$ meets every interval of length $M$. Then $\lim \inf_{i\to\infty}|D\cap s_i|/|s_i|$ will be at
least $1/M$, so $D$ cannot have density 0.} Therefore, for each  $\epsilon>0$
there is an  integer $m\not\in B$ satisfying ($**$).

Fix a sequence of integers $m_i\not\in B$ such
that the point $(\gamma_1(m_i)\ldots \gamma_{n-1}(m_i))$ tends to
$\bar{s}$.

Define $V\subseteq\bbz^{n+1}$ to be the set of all points 
$$\left(m_i,\lfloor g_1(m_i)\rfloor, \lfloor(
g_2(m_i,\gamma_1(m_i))\rfloor,\ldots \lfloor g_n(m_i,\gamma_1(m_i)\ldots\gamma_{n-1}(m_i))\rfloor\right)$$
for $i=1,2\ldots$. 

The equation $h(\gamma_1(m_i)\ldots\gamma_n(m_i))=0$ holds for all
$i$. Therefore, the  equation

$$h(\sigma_1(y_0)-y_1\ldots \sigma_n(y_0\ldots y_{n-1})-y_n)=0$$ holds
for all points $(y_0\ldots y_n)\in V$.  
Let $H(\bar{y})$ denote the polynomial on the left of the
above expression, so $H(\bar{y})$ has algebraic
coefficients and vanishes on $V$.

We claim that $H$ must have a non-constant factor with rational
coefficients. We shall prove this by arguing that the Zariski closure
of $V$ over the complex numbers includes a hypersurface in $\bbc^{n+1}$. The vanishing
ideal of that hypersurface will be principal, and defined over $\bbq$,
hence generated by a rational polynomial. That rational polynomial
will be a divisor of $H$.

To proceed, choose an infinite subset $V_0$ of $V$ such that the
Zariski closure $Z$ of $V_0$ is an irreducible component of the
Zariski closure of $V$.  We will show that $Z$ is a hypersurface in
$\bbc^{n+1}$ by arguing that no non-zero complex polynomial $k(y_0\ldots
y_{n-1})$ vanishes on $V_0$. 

Just suppose that $k(y_0\ldots y_{n-1})$ did vanish on $V_0$. Since
$V_0\subset\bbr^{n+1}$, we can assume that $k$ has real coefficients. Since
the coordinates of $\bar{s}$ are algebraically independent, it follows
that $k(t,g_1(t)-s_1,\ldots, g_{n-1}(t,s_1\ldots s_{n-2})-s_{n-1} )$ is
not the zero polynomial. Write 

 \begin{equation*}k(t,g_1(t)-s_1,\ldots, g_{n-1}(t,s_1\ldots s_{n-2})-s_{n-1}
)=\sum_{i=1}^Lk_i(\bar{s})t^i\end{equation*} with $k_L(\bar{s})\ne 0$. Choose a
neighborhood $U$ of $\bar{s}$ on which  $k_L(\bar{x})$ is bounded away
from zero. Then we can choose $M$ so large that for $t>M$ and for $\bar{x}\in U$,
it holds that \begin{equation*} k(t,g_1(t)-x_1,\ldots, g_{n-1}(t,x_1\ldots
x_{n-2})-x_{n-1} )\ne 0.\tag{$***$}\end{equation*}

 Now choose $i$ so that
\begin{enumerate}
\item $m_i>M$.
\item $(\gamma_1(m_i)\ldots\gamma_{n-1}(m_i))\in U.$
\item $\left(m_i,\lfloor
g_1(m_i)\rfloor, \lfloor(g_2(m_i,\gamma_1(m_i))\rfloor,\ldots \lfloor g_n(m_i,\gamma_1(m_i)\ldots\gamma_{n-1}(m_i))\rfloor\right)\in
V_0.$
\end{enumerate}
Substituting
$\gamma_1(m_i)\ldots\gamma_{n-1}(m_i)$ for $x_1\ldots x_{n-1}$ and
also 
$m_i$ for $t$ in ($***$)
 we obtain
$$ k(m_i,g_1(m_i)-\gamma_1(m_i),\ldots,g_{n-1}(m_i,\gamma_1(m_i),\ldots,
\gamma_{n-2}(m_i))-\gamma_{n-1}(m_i) )\ne 0.$$
Looking at the definition of the $\gamma_i$, we see
that the above inequation is equivalent to
$$k(m_i,\lfloor g_1(m_i)\rfloor,\lfloor( g_2(m_i,\gamma_1(m_i))\rfloor,\ldots,\lfloor
g_n(m_i,\gamma_1(m_i),\ldots,\gamma_{n-1}(m_i))\rfloor)\ne 0.$$
But this is a contradiction, because the point $$(m_i,\lfloor g_1(m_i)\rfloor,\lfloor( g_2(m_i,\gamma_1(m_i))\rfloor,\ldots,\lfloor
g_n(m_i,\gamma_1(m_i),\ldots,\gamma_{n-1}(m_i))\rfloor)$$ is an
element of $V_0$, hence $k$ vanishes at this point.
We conclude that $Z$, which is the Zariski closure of $V_0$, is a
hypersurface in $\bbc^{n+1}$.

The vanishing ideal $I\subseteq \bbc[\bar{y}]$ of $Z$ is
 principal.   Since $Z$ is the Zariski closure of a set of
 points with integer coordinates, $I$ has a generator $Q$ in
 $\bbq[\bar{y}]$. The polynomial $Q$ is the divisor of $H$ that we
 were after.

To complete the proof, suppose $H$ factors as $Q\cdot P$. Then the
coefficients of $P$ are real algebraic numbers, and we have the
following equality of polynomials:
$$Q(\bar{y})\cdot P(\bar{y})=h(\sigma_1(y_0)-y_1\ldots \sigma_n(y_0\ldots y_{n-1})-y_n)$$
Substituting $g_i(t,\bar{s})-r_i$ for $y_i$ ($i=1\dots n$) in the
last equation, we obtain $$A\cdot B= h(r_1\ldots r_n),$$ where

$$A=Q(t,g_1(t)-r_1,\ldots,g_{n}(t,r_1\ldots r_{n-1})-r_n)$$
and $$B=P(t,g_1(t)-r_1,\ldots,g_{n}(t,r_1\ldots r_{n-1})-r_n).$$
Working in the ordered ring $$\bbq[t,g_1(t)-r_1,g_2(t,r_1)-r_2,\ldots
g_n(t,r_1,\ldots,r_{n-1})-r_n],$$ we have that $A\cdot B$ is finite,
and neither  is infinitesimal, therefore both are
finite. But $A$ has the form $A_1/n$, where $A_1$ is a polynomial
with integer coefficients. Thus $A_1$ is a finite transcendental element of  $\ugh$. But then $\ugh$ is not
discretely ordered, contrary to our assumption that $\ugh$ extends to
a model of open induction. 
\end{proof}
\begin{remark}
Theorem 4.1 is almost certainly not giving the whole truth. We believe
that a ring $\ugh$ generated by a special sequence is diophantine
correct if and only it extends to a model of open induction, with no
restrictions on the tuple $\bar{r}$ beyond algebraic independence.  We
also believe that a theorem like Theorem 4.1 holds for the more
general sequences of Puiseux polynomials used to construct models of open
induction in \cite{Be}. To prove this, one must extend the results
of \cite{Bl} to an appropriate class of ``generalized'' semialgebraic
functions, that is, compositions of semialgebraic functions with the
integer part operator. 
\end{remark}


\begin{thebibliography}{Abc}

\bibitem{Be} Berarducci, A. and Otero, M.  \emph{A recursive
nonstandard model of normal open induction,}  Journal of Symbolic
Logic  \textbf{61} (1996), no. 4, 1228-1241.

\bibitem{Bl} Bergelson, V. and Leibman, A. \emph{ Distribution of
Values of Bounded Generalized Polynomials,} 
Acta. Math. \textbf{198} (2007), no. 2,  155-230.

\bibitem{Bo} Borevich, B. I. and Shafarevich, I. R, \emph{ Number
Theory,} Academic Press, New York, 1966.

\bibitem{Br} Brumfiel, G. \emph{Partially Ordered Rings and Semi-Algebraic
Geometry,} Cambridge University Press, Cambridge, 1979.

\bibitem{Ca} Cassels, J.W.S \emph{ Diophantine
Approximation,} Cambridge University Press, Cambridge, 1957.

\bibitem{Mo} Mordell, L.J. \emph{ Diophantine Equations,} Academic
Press, London, 1969.

\bibitem{Sh} Shepherdson, J. \emph{ A Nonstandard Model for a Free Variable
Fragment of Number Theory,} Bull. L'Acad, Pol. Sci. \textbf{112} (1964), 79-86.
  
\bibitem{Vd} Van den Dries, L. \emph{ Tame Topology and O-Minimal Structures,} 
Cambridge University Press, Cambridge, 1998.

\bibitem{Vc} Van der Corput, J. G. \emph{ Diophantische
Ungleichungen II. Rythmische Systeme A, B,} 
Acta. Math. \textbf{159} (1932), 209-328.

\bibitem{Wi} Wilkie, A. \emph{ Some Results and Problems on Weak
Systems of Arithmetic,} Logic Colloquium '77 (A. Macintyre
 \textit{ et al,} eds.), North-Holland, 1978, pp. 285-296.
 





  
 


\end{thebibliography}
\end{document}